\documentclass[a4paper,11pt]{article}
\usepackage[latin1]{inputenc}
\usepackage[T1]{fontenc}
\usepackage[british,UKenglish,USenglish,american]{babel}
\usepackage{amsmath}
\usepackage{amsfonts}
\usepackage{hyperref}
\usepackage[T1]{fontenc}
\usepackage[latin1]{inputenc}
\usepackage{ntheorem}
\usepackage{theorem}
\usepackage{graphics}
\usepackage{makeidx}
\usepackage{fancyhdr}
\usepackage{hyperref}
\usepackage{verbatim}
\usepackage{bbm}
\usepackage{bm}
\usepackage{amssymb}
\usepackage{color}
\numberwithin{equation}{section}
\usepackage{amssymb}
\newcommand{\be}{\begin{equation}}
\newcommand{\ee}{\end{equation}}
\newcommand{\benn}{\begin{equation*}}
\newcommand{\eenn}{\end{equation*}}
\newcommand{\bea}{\begin{eqnarray}}
\newcommand{\eea}{\end{eqnarray}}

\newcommand{\beann}{\begin{eqnarray*}}
\newcommand{\eeann}{\end{eqnarray*}}

\newtheorem{theorem}{Theorem}[section]

\newtheorem{corollary}[theorem]{Corollary}
\newtheorem{lemma}[theorem]{Lemma}

\newtheorem{definition}[theorem]{Definition}
\theorembodyfont{\rm}

\newtheorem{remark}[theorem]{Remark}

\newtheorem{assumptions}[theorem]{Assumption}

\newcommand{\qed}{\hfill $\Box$\smallskip}
\newcommand{\E}{\noindent{$\mathbb{E}$ \ }}
\DeclareMathOperator{\sgn}{sgn}

\def\R{\mathbb{R}}

\def\N{\mathbb{N}}

\def\P{\mathbb{P}}
\def\E{\mathbb{E}}

\def\P{\mathbb{P}}

\def\cF{\mathcal{F}}

\def\cL{\mathcal{L}}

\def\cN{\mathcal{N}}
\def\cO{\mathcal{O}}

\def\cS{\mathcal{S}}

\def\txtd{{\textnormal{d}}}

\def\txtD{{\textnormal{D}}}

\allowdisplaybreaks

\title{On the approximation of finite-time Lyapunov exponents for the stochastic Burgers equation}
\author{Alexandra Blessing Neam\c tu~\thanks{Alexandra Blessing Neam\c tu.  Department of Mathematics and Statistics, University of Konstanz,
Universit\"atsstra\ss{}e~10, 78464 Konstanz, Germany. E-Mail: alexandra.blessing@uni-konstanz.de}~~~~and~~~~Dirk Bl\"omker\thanks{Dirk Bl\"omker. Institut f\"ur Mathematik, Universit\"at Augsburg, Universit\"atsstra{\ss}e 12, 86135 Augsburg, Germany.~E-Mail: dirk.bloemker@math.uni-augsburg.de} }
\begin{document}
\maketitle

\begin{abstract}
We analyze stochastic partial differential equations (SPDEs) with quadratic nonlinearities close to a change of stability / phase-transition.~To this aim we compute finite-time Lyapunov exponents (FTLEs), observing a change of sign based on the interplay between the distance towards the bifurcation and the noise intensity. A technical challenge is to provide a suitable control of the nonlinear terms coupling the dominant and stable modes of the SPDE and of the corresponding linearization.~In order to illustrate our results we apply them to the stochastic Burgers equation.

\end{abstract}

\section{Introduction}

Finite-time Lyapunov exponents are a very useful tool to detect local stability/instability of a stochastic system.~Negative finite-time Lyapunov exponents indicate attraction whereas positive ones indicate that nearby solutions tend to separate on a finite time horizon.~These have been investigated for SDEs~\cite{AK:84, CDLR17,DELR:18,BBBE:25} and SPDEs with cubic nonlinearities such as Allen-Cahn or Swift-Hohenberg in~\cite{BlEnNe:21,BlNe:23,BB:25,B:25}.
Quadratic nonlinearities are significantly more difficult, as cubic nonlinearities often add additional stability and dissipation to the system.~To our best knowledge, this is the first work that investigates finite-time Lyapunov exponents (FTLE) for SPDEs with quadratic nonlinearities, such as the stochastic Burgers equation where the tools developed in \cite{BlEnNe:21,BlNe:23,BB:25} break down.~The quadratic nonlinearity introduces a coupling between the dominant and stable modes of the SPDE and its linearization which does not occur for stable cubic nonlinearities.~Therefore removing the stable modes using averaging techniques and getting closed equations for the dominant modes is challenging and requires a suitable It\^o trick together with a careful analysis of the coupling.   \\

In order to study a bifurcation via Lyapunov exponents, we consider equations close to a change of stability, where the theory of amplitude equations (AE) allows to reduce the dynamics to an SDE on the dominating modes.~This approximation is a well established tool and many results were published on this topic such as~\cite{BlHa:04, Bl:07, BHP, BlMo, Fu}.~Most of these results use an approximation over a long but  finite time-scale, which is the natural slow time-scale on which the interesting dominating dynamics evolve.~There are also results available for approximation of invariant measures via AE, see for example~\cite{BlHa:04}.~Nevertheless, an approximation of Lyapunov-exponents, which takes into account the limit of time to infinity, seems to be out of reach at the moment.~Therefore we focus on FTLE, since a change of sign in the FTLEs  detects a transition which is not the case for asymptotic Lyapunov exponents~\cite{CDLR17}. \\

In an informal way our main result, Theorem~\ref{scale:ftle} approximates the FTLEs for the full SPDE via the FTLE of the AE.
In contrast to our previous work \cite{BlNe:23} we present a fairly general approximation result in Theorem \ref{thm:lyapappr} that reduces the error term  between the finite-time Lyapunov exponents of the SPDE and the amplitude SDE to questions of stability.~We believe that this abstract result can be very useful for numerical purposes.~We plan to address this aspect in a future work for the stochastic Burgers equation exploiting the techniques in~\cite{BlJ}. 
In this work we need precise error analysis of the approximation for amplitude equations (see Theorem \ref{thm:approximation}) and the corresponding linearizations (see Theorem \ref{thm:linear}).~This allows us to state implications on the bifurcation (positive or negative finite time Lyapunov exponent) once we know the FTLE for the AE. While our main approximation result is fairly general, we also need results on FTLEs for AE.~Let us remark that for 1D SDEs, there are many results for FTLE especially with additive noise~\cite{CDLR17,B:25}.
The key point is that the solution of the linearized equation is usually given as an explicit exponential, which allows for direct computations.~For FTLEs in 2D, we plan to use the techniques in~\cite{DELR:18,BBBE:25} and also to exploit symmetries which might help to further reduce the dimension.\\

Let us state something about the choice of the noise, i.e.~it is essential that the noise is additive for the analysis, while the approximation via amplitude equations hold in general also for multiplicative noise~\cite{Fu}.~However, the framework for FTLEs is well-established for additive noise, while results for SDEs with multiplicative noise are known, see e.g.~\cite{AK:84} for an example on compact manifolds,
the general theory for FTLEs for SDEs with multiplicative noise seems to be open to our knowledge.~Nevertheless, it is important to have small noise while we are close to a change of stability, in order for the approximation via AE to work. Otherwise, a too strong noise will spoil the features of the bifurcation that we want to exploit.~But let us remark that it is not essential to force the dominant modes at all. In that case our the approximation via AE still works and the AE is deterministic.~Nevertheless, if the dominant modes are unforced, one can allow for larger noise than we do in our result here and still obtain an approximation result for AE. See for example Bl\"omker, Hairer, Pavliotis~\cite{BHP}, or Mohammed, Klepel, B\"omker \cite{MBKa}, where additional terms appear in the AE, in particular multiplicative noise.~This complicates the analysis of the FTLEs significantly.\\

In our examples, we mainly focus on a Burgers-type equation with a one-dimensional kernel, but this is mainly due to the availability of results on FTLE for 1D SDEs.~We can also treat equations like 
Navier-Stokes, although there is no bifurcation, and no reduction to a dominant space $\cN$.~Here we could focus on the onset of convection in Rayleigh-B\'enard, which consists of a Navier-Stokes equation coupled to a heat equation.~The operator in this model is not self-adjoint and therefore we have to 
modify our approach, but this does not seem to be out of reach and will be treated in a future work, see Subsection \ref{sec:rb} for a brief discussion.~Moreover, models like Kuramoto-Sivashinsky or KPZ-type equations would fall into the scope of our result.~Another example is a surface growth model in~\cite{BlMo} with periodic boundary conditions where the kernel is two-dimensional, here in some cases we might exploit symmetries like polar coordinates to reduce the AE from 2D to 1D. The combination of cubic and quadratic nonlinearities in the model are also possible. Moreover higher order terms in the nonlinearity of the SPDE would not affect the approximation via AE too much.~In order to focus on the main features and challenges posed by the quadratic nonlinearity, and not get overwhelmed by technical details  we treat only quadratic nonlinearities here in this paper.  \\

While we work close to the bifurcation let us comment on the stable case below the threshold of stability.~For cubic nonlinearities due to additional dissipation, 
below the threshold one can show that FTLE exponents are negative once we are in the stable regime.~For Burgers or Navier-Stokes this is quite different, here the linear part has to provide sufficient dissipation, in order to verify that FTLEs are actually negative, as the linear part has to dominate the nonlinearity in the estimates.~This is an old well-known problem, already present in results on random attractors or synchronization for Navier-Stokes, for sufficiently strong dissipation the attractor is a singleton.~Using a similar analysis one can show that FTLE are negative in that case, but we do not comment in detail on this straightforward argument.~A more recent example is the work by Gess, Liu, Schenke~\cite{GLS:20} which applies also for the existence of a random attractor for the stochastic Burgers equation for large enough linear dissipativity.  \\

In a forthcoming work we plan to investigate other noise such as fractional Brownian motion based on the approximation results developed in~\cite{BlNe:22} and FTLEs in~\cite{BB:25} for SPDEs with stable cubic nonlinearities.~Note that the setting of this work is not restricted to Markovian noise.~However, we rely on an It\^o formula to eliminate certain terms of higher order.~Nevertheless, we believe that this can be replaced by other methods.~Furthermore, we also plan to extend our results for the rough Burgers equation with multiplicative noise exploring the rough path approach developed by~\cite{HaWe}. 

\paragraph{Plan of the paper} This work is structured as follows. In Section~\ref{sec:a} we introduce the setting and state the assumptions on the coefficients of the SPDE we consider. In Section~\ref{sec:approx} we state the approximation of the SPDE with an SDE called amplitude equation.~A technical tool is to combine a multiscale approach with a suitable It\^o trick in order to eliminate certain terms of higher order and to get a closed formula for the SDE describing the essential dynamics of the infinite-dimensional system. In order to compute FTLEs, we have to consider in Section~\ref{sec:main} the linearization of the SPDE.~A main technical challenge is to control large terms appearing in this linearization.~This effect is due to the quadratic nature of the nonlinearity and does not occur for cubic nonlinearities as treated in~\cite{BlNe:23}.~Therefore we have to develop different tools combining a multiscale approach with stopping time arguments and It\^o's formula
to rigorously handle such terms.~Thereafter we state our main results, Theorems~\ref{thm:lyapappr} and Theorem~\ref{scale:ftle}.~These provide precise error bounds for the difference between the finite-time Lyapunov exponents of the SPDE and those of the SDE. These are independent of the structure of the SDE and therefore can be used in different situations.~Based on these bounds, we provide a bifurcation analysis in Section~\ref{sec:bif}.~To this aim, we explore the fact that the drift of the SDE is given by a stable cubic nonlinearity for which we can use the previous results in~\cite{BlNe:23}.~We apply this setting  to the stochastic Burgers equation in Section~\ref{sec:app}.

\subsection*{Acknowledgments}
A. Blessing acknowledges funding by the DFG CRC/TRR 388 "Rough Analysis, Stochastic Dynamics and Related Fields" - Project ID 516748464 and DFG CRC 1432 " Fluctuations and Nonlinearities in Classical and Quantum Matter beyond Equilibrium" - Project ID 425217212.

\section*{Conflict of interest statement}
The authors have no conflicts of interest to declare. All co-authors have seen and agree with the contents of the manuscript and there is no financial interest to report. 
\section*{Data availability}
No data was used for the research described in the article.
\section*{Ethics declaration}
Not applicable. 

\section{Setting and Assumptions}\label{sec:a}

We let $X$ stand for a separable Banach space and consider the SPDE driven by a cylindrical Brownian motion $(W(t))_{t\geq0}$ 
\begin{align}\label{spde}
\begin{cases}
\txtd u = [A u + \nu u +B(u,u)]~\txtd t +\sigma \txtd W(t) \\
u(0)=u_{0}\in X.
\end{cases}
\end{align}

\begin{definition}\label{order}
For the $\cO$-notation here we use that an $X$-valued process $M$ is $\cO(f)$ for a term $f$ on a possibly random interval $I$, if for all probabilities $p\in(0,1)$ there is a constant $C_p>0$ such that $\P (\sup_{t\in I}\|M(t)\|_X \leq C_p f) \geq p$. For time independent quantities we use the similar notation without the supremum in time. 
If the process $M$ and the bound $f$ depends on some small quantity  $\varepsilon>0$ we assume that the constant $C_p$ is independent of $\varepsilon \in (0,\varepsilon_0]$ for some fixed $\varepsilon_0>0$. We also use the abbreviation that a process $M_\varepsilon$ is  $\cO(\varepsilon^{n-})$  if 
for all $\kappa>0$ we have $M_\varepsilon=\cO(\varepsilon^{n-\kappa})$.
\end{definition}

We make the following standard assumptions on the linear operator $A$, on the quadratic non-linearity $B$ and on the noise. 
\begin{assumptions}\label{a} (Differential operator $A$)
The linear operator $A$ generates a compact analytic semigroup $(e^{tA})_{t\geq 0}$ on $X$. 
 Moreover, it is symmetric, non-positive and has a finite-dimensional kernel  which we denote by $\cN$. 
 We define  the projection  $P_c$ onto $\cN$, set $P_s=\text{Id}-P_c$ and obtain that $X=\cN\oplus \cS$, where $\cS$ stands for the range of $P_s$. 
 The semigroup is exponentially stable on $\cS$ which means that there exists $\mu>0$ such that 
\[ 
\|e^{tA} P_s\|_{\cL(X)}\leq e^{-t\mu},~~\text{ for all } t\geq 0. 
\]
We further assume that there exists another Hilbert space $Y$ such that $Y\hookrightarrow X$ and that we can extend the semigroup $(e^{tA})_{t\geq 0}$ to a semigroup on $Y$. Moreover, we assume that for every $t>0$, $e^{tA}$ is a bounded operator from $Y$ to $X$ such that for some $\alpha\in[0,1)$ we have
\[ \|P_s e^{tA}y\|_X\leq M (1+t^{-\alpha}) \|y\|_Y, \text{ for all } t>0 \text{ and }y\in Y.  \]

\end{assumptions}


\begin{assumptions}\label{b}(Nonlinearity)
We assume that $B:X\times X\to Y$ is a bilinear continuous symmetric map, i.e. $B(u,v)=B(v,u)$ and there exists a constant $C_B>0$ such that
\[ \|B(u,v)\|_Y \leq C_B \|u\|_X \|v\|_X. 
\]
We set $B(u):=B(u,u)$, $B_s=P_s B$, $B_c=P_cB$ and assume that $B_c(\cdot,\cdot)=0$ on $\cN\times\cN$.~We further require that 
\[\cF_c(a):=-P_cB(a,A_s^{-1}B_s(a,a))= -B_c(a,A^{-1}_sB_s(a,a))\] is a stable cubic meaning that $\langle \cF_c(a),a\rangle \leq0$.
Moreover, we assume that for any positive $\delta>0$ there is a constant $C>0$ depending on $\delta$ such that for all $a, b\in \cN$ 
\begin{equation}\label{sign}
\langle\cF_c(a+b),a\rangle 
\leq -c\|a+b\|^4_{\cN} + C\|a+b\|^3_\cN\|b\|_\cN
\leq - \delta\|a\|^4_\cN + C_\delta \|b\|^4_\cN.
\end{equation}
We additionally assume that $\langle D \cF_c(a)b, b \rangle\leq 0$. 
\end{assumptions}

\begin{assumptions}\label{n}(Noise)
 We assume that $(W(t))_{t\geq 0}$ is a cylindrical Wiener process on a probability space $(\Omega,\mathcal{F},\mathbb{{P}})$ with covariance operator $Q\in \cL(X)$. Moreover, $W_c:=P_cW$ is a Wiener process on $\cN$ with covariance operator $P_cQ P^*_c$ and $W_s:=P_sW$ is a Wiener process on $\cS$ with covariance operator $P_s Q P^*_s$. We further assume that $W_c$ and $W_s$ are independent and that $A^{-1}_s W_s$ is a Wiener process in $X$ with covariance operator $A_s^{-1}P_s Q P_s^* (A_s^{-1})^*$. Furthermore, for an arbitrary basis $(e_k)_{k\in\N}$ of $X$ we assume that
\begin{align}\label{sum:c}
    \sum\limits_{k=1}^\infty \|Q^{1/2}_c e_k\|^2_X<\infty
\end{align}
respectively
\begin{align}\label{sum:s}
   \sum\limits_{k=1}^\infty B_c (P_c Q^{1/2}e_k, A^{-1}_sP_s Q^{1/2} e_k)<\infty.   
\end{align}




\end{assumptions}

\begin{assumptions}\label{o:z}
The stochastic convolution
\[ Z(t):=\int_0^t e^{A(t-s)}~\txtd W(s)\]
is well-defined and has $\P$-a.s.~continuous trajectories in $X$. We further have for every small $\kappa>0$
\begin{align}\label{scaling:z}Z_s(T):=P_s Z(T) = \cO(T^{\kappa})\qquad\text{and}\qquad 
Z_c(T):=P_c Z(T) = P_c W(T) = \cO( T^{1/2})
\end{align}
on any $[0,T]$ with $T>0$ in the space $X$. The fact that $P_sZ(T)=\cO(T^\kappa)$ follows by the well-known factorization method~\cite{DaPZ:92}.  
 By Chebyshev's inequality, one obtains for $p\geq 1$ the existence of a constant $C_p>0$ such that 
    \[ \P(\sup\limits_{t\in[0,T_0]} \|Z_s(t)\|_X \geq \delta )\leq C_p \delta^{- p } T_0.    \]
    This probability can be made arbitrarily small taking $p$ large enough.~For more details, see~\cite[Remark B.1]{BlNe:22}. 
\end{assumptions}

%


\subsection{Finite-time Lyapunov-exponents}\label{sec:ftle}

The linearization $\txtD_{u_0} u(t,\omega,u_0)$ of~\eqref{spde} around a solution $u(t,\omega,u_0)$ with initial condition $u_0$ is defined as the solution $v(t,\omega,u_0,v_0)$ of the linear PDE called also the variation equation, which due to the additive structure of the noise and the quadratic nonlinear term is given by
\begin{align}\label{linearization}
\begin{cases}
    \txtd v =  [Av + \nu v + 2B(u, v)]~\txtd t \\
    v(0)=v_0.
    \end{cases}
\end{align}
\begin{remark}
The Fr\'echet differentiability of the solution operator $u_0 \ni X \mapsto u(t,\omega,u_0)\in X$ follows subtracting the stochastic convolution $Z$ from the SPDE~\eqref{spde}.~Therefore we obtain a random PDE to which we can apply pathwise deterministic regularity results~\cite[Theorem 3.4.4]{Henry81}.
\end{remark}

For $t>0$ we denote the random solution operator $U_{u_0}(t):X\to X$ such that $v(t)=U_{u_0}(t)v_0$, where $v$ is a solution of \eqref{linearization} given the initial condition $v_0\in X$ and define finite-time Lyapunov exponents as in~\cite{BlEnNe:21}.

\begin{definition}{\em (Finite-time Lyapunov exponent)}. Let $t>0$ be fixed.~We call a finite-time Lyapunov exponent for a solution $u$ of the SPDE with  initial condition $u_0=u_0(\omega)$ 
\begin{equation}\label{ftle}
    \lambda_t(u_0):=\lambda(t,\omega,u_0(\omega)) =\frac{1}{t} \ln \left( \| U_{u_0}(t) \|_{\cL(X)}\right).
\end{equation}
\end{definition}

Based on this we compute the FTLEs as in~\cite{BlNe:23} as follows. 
\begin{remark}
\label{propLyap}
We can compute $\|U_{u_0}\|_{\cL(X)}$ as follows
\begin{eqnarray*}
\| U_{u_0}(t) \|_{L(X)}
&=& \sup\{ \|v(t)\| / \|v(0)\| \ : \  v \text{ solves \eqref{linearization} with }v(0)\not=0 \} \\
&=& \sup\{ \|v(t)\|  \ : \  v \text{ solves \eqref{linearization} with }\|v(0)\|=1 \}.
\end{eqnarray*}
\end{remark}

\section{Approximation with amplitude equations}\label{sec:approx}
Let us first comment on the concept of solutions. Under our assumptions, it is straightforward and well known (see \cite{BlMo} for example in a similar setting) to show the existence of local solutions, which is sufficient for our results. 

\begin{definition}
    A maximal local solution of~\eqref{spde} is an $X$-valued continuous process $u$ defined for times up to stopping time $\tau^\text{ex}$  satisfying the variation of constants formula for $t\in [0,\tau^\text{ex})$
    \begin{align}\label{mild:sol}
        u(t)= e^{tA} u_0 + \nu \int_0^t e^{A(t-s)}u(s)~\txtd s + \int_0^t e^{A(t-s)}B(u(s),u(s))~\txtd s + \sigma Z(t),~~t>0, 
    \end{align}
    such that with probability one either $\tau^\text{ex}=\infty$ or $u(t) \to \infty $  for $t\nearrow \tau^\text{ex}$ in $X$.
\end{definition} 

Close to a change of stability, we reduce the infinite dimensional dynamics of the SPDE~\eqref{spde} to an SDE.~To this aim we fix the following setting. 

\begin{assumptions}{(Approximation)}
We make the following assumptions:
    \begin{itemize}
        \item initial condition: $P_cu(0) =\cO(\varepsilon)$ and $P_s u(0)=\cO(\varepsilon^2)$. 
        \item the parameter $\nu$ indicating the distance to the bifurcation and the noise intensity $\sigma$ satisfy \begin{align}
    \sigma \varepsilon^{-2} \leq C, \quad \nu \varepsilon^{-2}\leq C,
\end{align}
for an arbitrary constant $C>0$.
    \end{itemize}
\end{assumptions}

{\bf Ansatz.} For $U_c \in\cN$ and $ U_s\in\cS$ we make the ansatz $U=U_c+\varepsilon U_s$ and obtain on the slow time scale $T=\varepsilon^2 t $
\[ u(t) =\varepsilon U(\varepsilon^2t)= \varepsilon U_c(\varepsilon^2t) + \varepsilon^2 U_s (\varepsilon^2t). 
\]
Therefore we get that
\begin{align}
&\txtd U_c = [\nu\varepsilon^{-2} U_c +\varepsilon^{-1} B_c(U_c+\varepsilon U_s)]~\txtd T + \varepsilon^{-2}\sigma \txtd  \tilde{W}_c(T)\label{Uc}\\
&\txtd U_s = \varepsilon^{-2}[A_s U_s + \nu U_s + B_s( U_c+\varepsilon U_s)]~\txtd T +\sigma\varepsilon^{-3} \txtd \tilde{W}_s(T), \label{Us}
\end{align}
where $(\tilde{W}_c(T))_{T\in[0,T_0]}$ and $(\tilde{W}_s(T))_{T\in[0,T_0]}$ are rescaled Brownian motions for a fixed $T_0$.~The first main goal is to show that we can remove  $B_c(U_c,U_s)$ from the equation~\eqref{Uc} and show that~\eqref{Uc} approximates the dynamics of the SPDE~\eqref{spde} up to a small error  term.\\

To this aim, we first let the radius $r_c >0$ be large and the exponent $\kappa$ arbitrarily small, and introduce the stopping time 
\begin{align}\label{stopping:time}
   \tau^\star:=\inf\{T\in[0,T_0] : \|U_c(T)\|_X \geq r_c, \|U_s(T)\|_X \geq \varepsilon^{-\kappa}\}.  
\end{align}



\begin{remark}\label{rem:ou} Due to Assumption~\ref{o:z} the rescaled Ornstein-Uhlenbeck process $\tilde{Z}_\varepsilon$ 
defined as follows satisfies 
    \begin{align}\label{rescaled:ou}
\tilde{Z}_\varepsilon(T):= P_s\int_0^T e^{\varepsilon^{-2}A(T-S)}~\txtd \tilde{W}(S)= \varepsilon\int_0^{T\varepsilon^{-2}} e^{A(\varepsilon^{-2}T-s)}~\txtd {W}(s) = \varepsilon Z_s(T\varepsilon^{-2}) =\cO(\varepsilon^{1-}),
\end{align}
for a rescaled Brownian motion $(\tilde{W}(T))_{T\geq 0}$ with $\tilde{W}(T) = \varepsilon W(T\varepsilon^{-2})$. %
\end{remark}

The next result provides an upper bound on $U_s$ in terms of $r_c$ up to the stopping time $\tau^*$.
\begin{lemma}\label{us:mild} Let $\varepsilon_0>0$ be fixed and $P_su_0 =\cO(\varepsilon^{2})$ for $\varepsilon\in(0,\varepsilon_0]$. Then $\|U_s\|_X=\cO(\varepsilon^{0-}) $ on $[0,\tau^\star]$. 

\end{lemma}

\begin{proof}
We use the mild formulation 
\begin{align}\label{us:variation}
U_s(T) &= e^{\varepsilon^{-2}TA_s} U_s(0) + \varepsilon^{-2}\int_0^T e^{\varepsilon^{-2}(T-S)A_s}  [\nu U_s(S) + B_s(U_c(S)+\varepsilon U_s(S))]~ \txtd S\\
&+ \sigma\varepsilon^{-3} \tilde{Z}_\varepsilon (T). \nonumber
\end{align}
Thus, using Assumption~\ref{a}, Remark~\ref{rem:ou} and that $P_su_0=\cO(\varepsilon^2)$ consequently $U_s(0)=\cO(1)$, we obtain for $T\in [0,\tau^\star]$ that
\begin{align*}
    \|U_s(T)\|_X &\leq C \|U_s(0)\|_X  +C\varepsilon^{-2}\nu \int_0^T e^{-(T-S)\varepsilon^{-2}\mu}\varepsilon^{-\kappa}~\txtd S \\ &+ C \varepsilon^{2\alpha-2} \int_0^T (T-S)^{-\alpha} e^{-(T-S)\varepsilon^{-2}\mu} (r_c+\varepsilon^{1-\kappa})^2 \txtd S + \sigma \varepsilon^{-3} \|\tilde{Z}_\varepsilon(T)\|_X\\
    &=\cO(\varepsilon^{0-}) \quad \text{on } [0,\tau^\star].
\end{align*}
Taking $\varepsilon_0$ small proves the assertion keeping in mind that $\nu\varepsilon^{-2}\leq C$, $\sigma \varepsilon^{-2}\leq C$, $\|\tilde{Z}_\varepsilon(T)\|_X=\cO(\varepsilon^{1-})$.
\end{proof}

\paragraph{Heuristic sketch of the approach.}
In order to remove the term $B_c(U_c,U_s)$  from the equation~\eqref{Uc} we apply It\^o's formula to $B_c(U_c,A_s^{-1}U_s)$. We will show that 
\[
\int_0^T P_cB(U_c(S),U_s(S))~\txtd S
\approx 
-\int_0^T P_cB(U_c(S),A_s^{-1}B_s(U_c(S),U_c(S))~\txtd S
 \]
 up to an error term of order $\cO(\varepsilon)$. 
 The ansatz entails on $[0,\tau^\star]$ since $\varepsilon^{-2}\nu =\cO(1)$ and $\varepsilon^{-2}\sigma=\cO(1)$ that
 \begin{align}
 \txtd U_c(T) &= \cO(r_c+{\varepsilon^{-1}}(r_c+\varepsilon^{1-\kappa})^2) \txtd T +  \sigma\varepsilon^{-2} \txtd\tilde{W}_c(T) \nonumber \\
 &= \cO(r_c+{\varepsilon^{-1}}(r_c+\varepsilon^{1-\kappa})^2) \txtd T + \cO(1) \txtd\tilde{W}_c(T)
 \end{align}
 respectively
 \begin{align}\label{ito:us}
     \txtd U_s(T) &= \varepsilon^{-2}[A_sU_s +B_s(U_c)] \txtd T  + \cO(\varepsilon^{-1-\kappa}r_c+ \varepsilon^{-2\kappa}) \txtd T
 + \sigma \varepsilon^{-3} \txtd\tilde{W}_s(T) \nonumber\\ 
     &= \varepsilon^{-2}[A_sU_s +B_s(U_c)] \txtd T  + \cO(\varepsilon^{-1-\kappa}r_c+ \varepsilon^{-2\kappa}) \txtd T
 + \cO(\varepsilon^{-1}) \txtd\tilde{W}_s(T).
 \end{align}
  Thus It\^o's formula on $[0,\tau^*]$  entails
\begin{align}
    \txtd B_c(U_c,A_s^{-1}U_s)  
&= B_c(\txtd U_c,A_s^{-1} U_s) + B_c(U_c,A_s^{-1}\txtd U_s) + \sigma^2 \varepsilon^{-5}B_c(\txtd \tilde{W}_c,A_s^{-1}\txtd \tilde{W}_s) \nonumber\\ 
& = B_c( [\nu\varepsilon^{-2} U_c +\varepsilon^{-1}(B_c(U_c+\varepsilon U_s))]~\txtd T , A^{-1}_sU_s) + \varepsilon^{-2}\sigma B_c(\txtd \tilde{W}_c(T), A^{-1}_sU_s)\nonumber \\
& + B_c(U_c, A^{-1}_s [\varepsilon^{-2}A_sU_s+\nu U_s + B_c(U_c+\varepsilon U_s) ]~\txtd T) + \sigma\varepsilon^{-3} B_c(U_s,A^{-1}_s~\txtd\tilde{W}_s(T)) \nonumber\\
& = \cO(\varepsilon^{-1-\kappa})\txtd T+\cO(\varepsilon^{-\kappa})\txtd\tilde{W}_c + \cO(\varepsilon^{-\kappa}) \txtd\tilde{W}_s + \varepsilon^{-2} B_c(U_c,A_s^{-1}[A_sU_s+B_s(U_c)] \txtd T).\label{ito:approx}
\end{align}
Recalling that $W_c$ and $W_s$ are independent we therefore have that $B_c(\txtd \tilde{W}_c,A_s^{-1}\txtd \tilde{W}_s)=0$.
Since $B:X\times X\to Y$ and $A^{-1}_s$ is a bounded operator on $X$, all estimates of the terms above on $[0,\tau^*]$ are straightforward, see~\cite[Lemma 24]{BlMo} for computations in a similar setting. 

In conclusion collecting the terms of order $\cO(\varepsilon^{-2})$ we infer on  $[0,\tau^\star]$
\[ \int_0^T B_c(U_c, U_s) \txtd S= - \int_0^T B_c(U_c,A_s^{-1}B_s(U_c)) \txtd S + R_2(T),\]
where we will later show that $R_2$ is of order $\cO(\varepsilon^{1-\kappa})$ on $[0,\tau^*]$. Obviously, this term also depends on $r_c$.

\begin{lemma}{\em(Stochastic integrals)}\label{s:integrals}
Under the Assumption~\ref{n}, in particular~\eqref{sum:c} and~\eqref{sum:s}, the integrals 
\[ \int_0^T B_c(\txtd \tilde{W}_c(S), A^{-1}_s U_s(S)) \text{ and } \int_0^T B_c(U(S),A^{-1}_s~\txtd \tilde{W}_s(S))  \]
are well-defined and of order $\cO(\varepsilon^{-\kappa})$ respectively $\cO(r_c)$.
\end{lemma}

\begin{proof}
 The statement easily follows by the Burkholder-Davis-Gundy and Chebyshev's inequality.  
 Since $(\tilde{W}_c(T))_{T\geq 0}$ is a Wiener process on $\cN$ with covariance operator $P_c Q P^*_c$, we obtain by the Burkholder-Davis-Gundy inequality for $p>1$ 
\begin{align*}
    \E \sup\limits_{T\in[0,\tau^*]}\Big\| \int_0^T B_c(\txtd \tilde{W}_c(S), A^{-1}_s U_s(S)) \Big\|^p_Y &\leq C \Big(\E\int_0^{\tau^\star} \| B_c(\cdot, A^{-1}_s U_s(S)) Q^{1/2}_c\|^2_{\cL_2(X,Y)}~\txtd S\Big)^{p/2} \\
    & \leq C \E\Big( \int_0^{\tau^\star} \sum\limits_{k=1}^\infty\|B_c(Q^{1/2}_ce_k, A^{-1}_s U_s(S))\|^2_Y~\txtd S \Big)^{p/2}\\
    & \leq C  \Big(\sum\limits_{k=1}^\infty \|Q^{1/2}_c e_k\|^2_X\Big)^{p/2} \E \sup\limits_{T\in[0,\tau^{\star}]} \|A^{-1}_s U_s(T)\|^p_X.
\end{align*}

By Chebyshev's inequality we get for $c>0$ that 
\[ \P\Big(\sup\limits_{T\in[0,\tau^*]} \Big\|\int_0^T B_c(\txtd \tilde{W}_c(S), A^{-1}_s U_s(S) \Big\|^p_Y \geq c \Big)\leq C (c) \Big(\sum\limits_{k=1}^\infty \|Q^{1/2}_c e_k\|^2_X\Big)^{p/2} \E \sup\limits_{T\in[0,\tau^{\star}]} \|A^{-1}_s U_s(T)\|^p_X,     \]
meaning that 
\[\int_0^T B_c(\txtd \tilde{W}_c(S), A^{-1}_s U_s(S))
= \cO(\varepsilon^{-\kappa}).
\]

Similarly for the stochastic integral with respect to $\tilde{W}_s$, we have regarding that $(A^{-1}_s\tilde{W}_s(T))_{T\geq 0}$ is a Wiener process on $X$ with covariance operator 
$A_s^{-1}P_s Q P_s^* (A_s^{-1})^*$. Therefore 
\begin{align*}
    \E \sup\limits_{T\in[0,\tau^*]}\Big\| \int_0^T B_c(U_s(S), A^{-1}_s~\txtd \tilde{W}_s(S)) \Big\|^p_Y &\leq C\Big(\E \int_0^{\tau^\star} \Big\| B_c(U_c(S), A^{-1}_s Q^{1/2}_s) \Big\|^2_{\cL_2(X,Y)}~\txtd S \Big)^{p/2}\\
    & \leq \Big(\E \int_0^{\tau^\star} \sum\limits_{k=1}^\infty \|B_c(U_c(S), A^{-1}_s Q^{1/2}_se_k\|^2_Y~\txtd S \Big)^{p/2}\\
    & \leq C \E\sup\limits_{T\in[0,\tau^\star]}\|U_c(T)\|^p_X  \Big(\sum\limits_{k=1}^\infty \|A^{-1}_sQ^{1/2}_se_k\|^2_X\Big)^{p/2},
\end{align*}
which is finite by~\eqref{sum:s}. We notice
\begin{align*}
    \langle A^{-1}_s Q^{1/2}_s e_k, A^{-1}_s Q^{1/2}_s e_k  \rangle =\langle (A^{-1}_sQ^{1/2}_s)^* A^{-1}_s Q^{1/2}_s e_k, e_k  \rangle.
\end{align*}
Again, Chebyshev's inequality entails for $c>0$ that
\[ 
\P\Big(\sup\limits_{T\in[0,\tau^*]} \Big\| \int_0^T B_c(U_s(S), A^{-1}_s \txtd \tilde{W}_s(S)) \Big\|^p_Y \geq c \Big) \leq C(c,p) \E\sup\limits_{T\in[0,\tau^\star]}\|U_c(T)\|^p_X  \Big(\sum\limits_{k=1}^\infty \|A^{-1}_sQ^{1/2}_se_k\|^2_X\Big)^{p/2}. 
\]
In conclusion $\int_0^T B_c(U_s(S), A^{-1}_s\txtd\tilde{W}_s(S) )=\cO(r_c)$.
\end{proof}\\

 Later, we will remove the stopping time but first we generalize the previous statement. 

\begin{lemma}\label{o:integrals}
Let $A^{-1}_sW_s$ be a Wiener process in $X$. Then 
     $\int_0^T \cO(\varepsilon^k) \txtd A_s^{-1}W_s =\cO(\varepsilon^k)$ for $k\geq 0$.  
\end{lemma}
\begin{proof}
This follows exactly by the same arguments as Lemma~\ref{s:integrals}.
\qed
\end{proof}




Now we turn to the first approximation result for $U_c$ in $\cN$ up to the stopping time.
\begin{lemma}
    We have that $U_c$ solves 
the following amplitude equation  
\begin{align}\label{ae}
\txtd a = 
[\nu\varepsilon^{-2} a +2\cF_c(a)]~\txtd T + \varepsilon^{-2}\sigma \txtd P_c \tilde{W}_T 
\end{align}
up to small residual in integral form, i.e. 
\begin{align}\label{uc}
\txtd U_c = 
[\nu\varepsilon^{-2} U_c + 2\cF_c(U_c)]~\txtd T + \txtd R+\varepsilon^{-2}\sigma \txtd P_c \tilde{W}_T 
\end{align}
with $R=\cO(\varepsilon^{1-2\kappa})$ on $[0,\tau^\star]$ in $\cN$.
\end{lemma}

\begin{proof} 
We show that $R:=R_1+2R_2=\cO(\varepsilon^{1-2\kappa})$ on $[0,\tau^\star]$ in $\cN$.  

    Returning to~\eqref{Uc} we compute on $[0,\tau^*]$
    \[ \txtd U_c = [\nu\varepsilon^{-2} U_c + 2 B_c(U_c, U_s)]~\txtd T + \txtd R_1(T)+ \varepsilon^{-2}\sigma \txtd  \tilde{W}_c(T)
    \]
    with
    \begin{align*}
    R_1(T)= \int_0^T \varepsilon B_c(U_s,U_s)ds
        = \cO(\varepsilon^{1-2\kappa}),
    \end{align*}
    where we used that $B_c(U_c,U_c)=0$ and that $B_c(U_s,U_s)=\cO(\varepsilon^{-2\kappa})$ according to Lemma~\ref{us:mild}. Due to the It\^o trick above we have that 
    \[\int_0^T B_c(U_c,U_s)~\txtd S = - \int_0^T B_c(U_c, A^{-1}_s B_s(U_c) )~\txtd S + R_2(T)
     = \int_0^T \cF_c(U_c)~\txtd S + R_2(T),  \]
    where $R_2(T)$ is formally given by integrating~\eqref{ito:approx}. This entails
\begin{align*}
    R_2(T) =&\varepsilon^2 B_c(U_c(T), A^{-1}_s U_s(T)) \\&-  \varepsilon^2 B_c(U_c(0), A^{-1}_s U_s(0))-\int_0^T \cO(\varepsilon^{1-\kappa}) \txtd S - \int_0^T\cO(\varepsilon^{2-\kappa}) \txtd\tilde{W}_c -\int_0^T\cO(\varepsilon^{1-\kappa}) \txtd\tilde{W}_s. 
\end{align*}
Using that the stochastic integrals with respect to $\tilde{W}_c$ and $\tilde{W}_s$ are of order $\cO(\varepsilon^{-\kappa})$ respectively $\cO(1)$ as shown in Lemma~\ref{s:integrals}, one can conclude that $R=\cO(\varepsilon^{1-2\kappa})$ depending on $r_c$ on $[0,\tau^*]$.~The precise dependence on $r_c$ is not important for our aims, since we will remove the stopping time $\tau^*$ to conclude that $R=\cO(\varepsilon^{1-2\kappa})$ on $[0,T_0]$. \qed 
\end{proof}

\begin{lemma}\label{o:ae}
    The solution of the amplitude equation~\eqref{ae} is of order $\cO(1)$ on $[0,T_0]$. 
\end{lemma}
\begin{proof}
    This follows immediately from~\eqref{ae} using that $\nu\varepsilon^{-2}=\cO(1)$ respectively $\sigma\varepsilon^{-2}=\cO(1)$ together with the sign condition on $\cF_c$~\eqref{sign}. More precisely, subtracting $\varepsilon^{-2}\sigma \tilde{W}_c$ from~\eqref{ae} we obtain for $\tilde{a}=a-\varepsilon^{-2}\sigma$
    \[ \partial_T \tilde{a} =\nu \varepsilon^{-2} (\tilde{a}+\varepsilon^{-2}\sigma \tilde{W}_c) + 2\cF_c(\tilde{a} +\varepsilon^{-2}\sigma\tilde{W_c}  ). \]
Multiplying by $\tilde{a}$ we get for a constant $C>0$ that
\begin{align*}
    \frac{1}{2}\partial_T \|\tilde{a}\|^2&=\nu \varepsilon^{-2} \langle \tilde{a} +\varepsilon^{-2}\sigma \tilde{W}_c,\tilde{a}\rangle + 2\langle \cF_c (\tilde{a}+\varepsilon^{-2}\sigma\tilde{W}_c), \tilde{a}\rangle\\
    & \leq C \|\tilde{a}(T)\|^2 +C\|\tilde{W}_c(T)\|^4 - C \|\tilde{a}(T)\|^4.
\end{align*}
Gronwall's inequality yields for two arbitrary constants $c,C>0$ that
\[\|\tilde{a}(T)\|^2 \leq c\|\tilde{a}(0)\|^2 + C\int_0^T e^{c(T-S)}\| \tilde{W}_c(T)\|^4~\txtd S, \]
\end{proof}
proving the statement. \qed

\begin{theorem}\label{thm:approximation} 
Fix $T_0>0$. 
Let $u$ be a solution of the SPDE~\eqref{spde} such that $U_s(0)=\cO(1)$ and $U_c (0)=\cO(1)$. Then for all $p\in(0,1)$ and small $\kappa>0$ there exists a large constant $C_p$
and a set $\Omega_p$ with probability larger than $p$ 
such that on $\Omega_p$
\[ \sup\limits_{T\in[0,T_0]} \|U_c(T)\|  \leq C_p \qquad \sup\limits_{T\in[0,T_0]} \|U_s(T)\| \leq C_p \varepsilon^{-\kappa}.
\]

Moreover, if for some $c$, the $\cN$-valued process $a$ solves the amplitude equation~\eqref{ae} with initial condition such that $\|U_c(0)-a(0)\| \leq c\varepsilon$, then  on $\Omega_p$
\[ \sup\limits_{T\in[0,T_0]}\| U_c(T)- a(T)\| \leq C\varepsilon^{1-\kappa}\]
for $\varepsilon\in (0,\varepsilon_0]$ and some constant $C>0$ .
\end{theorem}

\begin{proof}
From Lemma \ref{lem:SDE-contdep} we know that  $U_c$ is close to the solution of the AE~\eqref{ae}, i.e. if $a(0)-U_c(0) =\cO(\varepsilon)$ and $a$ solves AE, then 
\[
a-U_c = \cO(\varepsilon^{1-}) \quad\text{on }[0,\tau^\star].
\]
Furthermore, Lemma~\ref{o:ae} shows that $a=\cO(1)$ on $[0,T_0]$. Furthermore we know that $a$ is bounded by $U_c$ up to $\tau^\star\leq T_0$ independently of $r_c$. Thus for every $p\in(0,1)$ we can choose a sufficiently large constant $C_p$, $r_c\geq C_p$ and $\varepsilon_0>0$ sufficiently small 
to obtain that $\tau^\star\geq T_0$ on a set of large probability $\Omega_p$, as $U_c$ remains bounded by a fixed constant. We now derive a bound for $U_s$ on the interval $[0,T_0]$. To this aim we use $B_s(U_c+\varepsilon U_s) = B_s(U_c,U_c) +2\varepsilon B_s(U_c,U_s) + \varepsilon^2 B_s(U_s,U_s)$ and the mild 
formulation~\eqref{us:variation}. Therefore we obtain that
    \begin{align*}
    \|U_s(T)\|_X &\leq e^{-T\varepsilon^{-2}\mu} \|U_s(0)\|_X + \varepsilon^{-2}\nu \int_0^T e^{-(T-S)\varepsilon^{-2}\mu} \|U(S)\|_X~\txtd S \\ &+ C \varepsilon^{2\alpha-2} \int_0^T (T-S)^{-\alpha} e^{-(T-S)\varepsilon^{-2}\mu} \|B_s(U_c(S)+\varepsilon U_s(S))\|_Y  \txtd S \\
    &+ \sigma \varepsilon^{-3} \|\tilde{Z}_\varepsilon(T)\|_X.
\end{align*}

Using the substitution $y=\frac{T-S}{\varepsilon^{-2}}$ to bound the last integral, results in 
\begin{align*}
    \varepsilon^{2\alpha-2} \int_0^T (T-S)^{-\alpha} e^{(T-S)\varepsilon^{-2}\mu}  \txtd S &= \int_0^{T\varepsilon^{-2}} e^{-\mu y} y^{-\alpha}  ~\txtd y 
    =\cO(1).
\end{align*}
Using that $U_s(0)=\cO(1)$, $\tilde{Z}_\varepsilon(T)=\cO(\varepsilon^{1-})$, $\nu \varepsilon^{-2}=\sigma\varepsilon^{-2}=\cO(1)$  and that  $\|U_c(T)\|=\cO(1)$ on ${[0,T_0]}$ by the first step, proves the statement. \qed

We also note for later use the following corollary, we can characterize the term, where the largest error contribution comes from.
\begin{corollary}\label{cor:lucky}
    Under the assumptions of Theorem \ref{thm:approximation} we have 
    \[\|U_s - \sigma^2\varepsilon^{-3}\tilde{Z}_\varepsilon\| =\cO(1)\quad\text{ on } [0,T_0].\]
\end{corollary}

\begin{remark}{(Attractivity)}
  Note that the assumptions on the initial condition $P_cu(0)=\cO(\varepsilon)$ and $P_su(0)=\cO(\varepsilon^2)$ are not restrictive. If $u(0)=\cO(\varepsilon)$ one can show that there exists a time $T_\varepsilon \approx c \ln(1/\varepsilon)$ such that $P_cu(0)=\cO(\varepsilon)$ and $P_su(0)=\cO(\varepsilon^2)$ after that we can apply the approximation result.~The proof of this statement relies on a modified treatment of the initial condition in the mild formulation providing appropriate bounds for $u_c$ and $u_s$. We refer to~\cite[Remark 18]{BlMo} for more details.

\end{remark}

\end{proof}


\section{Main result. Upper and lower bounds for finite-time Lyapunov exponents}\label{sec:main}


First we carefully analyze the linearization of the SPDE and then give an approximation result for the Lyapunov-exponents.
 
\subsection{Linearization}\label{sec:linearization}
On the slow time-scale $v(t)=\varepsilon V(\varepsilon^2 t)$
we linearize the SPDE~\eqref{spde} along an arbitrary solution $u$. Then on the slow time scale we consider $u(t)=\varepsilon U(t\varepsilon^2)$ and $v(t)=\varepsilon V(t\varepsilon^2)$.  Recalling that $\nu\varepsilon^{-2}\leq C$ and the fact that the nonlinearity $B$ is quadratic, we obtain
\begin{equation}\label{linear:v}
\begin{cases}
    \txtd V =  [\varepsilon^{-2}AV + \nu\varepsilon^{-2} V + 2\varepsilon^{-1} B(U, V)]~\txtd T 
    \\
    V(0)=\varepsilon^{-1}v_0.
    \end{cases}
    \end{equation}
Thereafter we split
\[V= V_c + V_s\]
and prove that with high probability  we have for a time $T_\varepsilon$ that 
 \[
     \sup\limits_{[0,T]} \|V_c\| = \cO(1) \qquad 
   {\sup\limits_{[T_\varepsilon,T]} \|V_s\| =\cO(\varepsilon)}.
\]

\begin{remark}
\begin{itemize}
    \item [1)] Note that in contrast to the previous ansatz $U=U_c+\varepsilon U_s$ we do not put an $\varepsilon$ here, as $V_s$ is allowed to be of order one at time $0$. 
\item [2)] Moreover for $V_s(0)=0$, or sufficiently small, we expect $T_\varepsilon=0$. Here we focus only on $V_s(0)\neq 0$ since the other case is simpler.
\item [3)] The term $\varepsilon^{-1}$ appearing in front of the nonlinearity in~\eqref{linear:v} is large and has to be handled by completely different tools than the ones in~\cite{BlNe:23} where this factor canceled due to the cubic structure of the nonlinearity. 
\end{itemize}    
\end{remark}
 
\begin{theorem} \label{thm:linear}
   Let $a$ be a solution of the AE~\eqref{ae}, $u$ be a solution of the SPDE~\eqref{spde}. Then for all $p\in(0,1)$ and small $\kappa>0$ there exists a constant $C_p$ and a set $\Omega_p$ with probability larger than $p$ such that
we have on $\Omega_p$ that 
\[\|V_s(T)\|\leq C_p \varepsilon, \quad \|V_c(T) -\varphi(T) \|\leq C_p\varepsilon^{1-\kappa} \quad \text{ for } T\in[T_\varepsilon,T], \] where $\varphi$ is the linearization of the AE~\eqref{ae} around an arbitrary solution $a$ satisfying the equation
 \[
\txtd \varphi = 
[\nu\varepsilon^{-2} \varphi +{2}B_c(\varphi,A_s^{-1}B_s(a,a))+{4}B_c(a,A_s^{-1}B_s(a,\varphi))]~\txtd T.
\]
\end{theorem}

\begin{proof}
Let us remark, that in principle we need no stopping time for $V$ since $U$ is bounded on a set with large probability, and the PDE for $V$ is linear. But as the system  for $V_c$ and $V_s$ is coupled, we introduce  for simplicity for $r_c$ and $r_s$ larger than one the stopping time
\begin{align}\label{stopping:timeV}
   \tau^v:=\inf\{T\in[0,T_0] : \|V_c(T)\|_X \geq r_c, \|V_s(T)\|_X \geq r_s\}. 
\end{align}

\paragraph{1st step} 
Using the mild formulation for $V_s$ we derive
\[
V_s(T)= e^{T\varepsilon^{-2} A_s} V_s(0)+ \int_0^T e^{(T-S)\varepsilon^{-2} A_s} [\nu\varepsilon^{-2} V_s(S)+ 2\varepsilon^{-1} B_s(U(S),V(S))]~\txtd S. 
\]
To compute the last term we use the ansatz $U=U_c+\varepsilon U_s$ and $V=V_c+V_s$ to get
\[
B_s(U,V)=B_s(U_c,V_c) +B_s(U_c,V_s) + {\varepsilon}B_s(U_s,V_c) + {\varepsilon} B_s(U_s,V_s).
\]
We further use that $U_c=\cO(1)$, 
and $U_s=\cO(\varepsilon^{0-})$ on the slow time scale
 $[0,T_0]$ from Theorem~\ref{thm:approximation} and the definition of the stopping time $\tau^v$ to obtain 
 \[
\|B_s(U,V)\|_Y \leq \cO(1)(r_c+r_s) + \cO(\varepsilon^{1-})(r_c+r_s).
\]
Therefore we infer on $[0,\tau^v]$ that
\begin{align}
\|V_s(T)\| &\leq  e^{- \mu T\varepsilon^{-2}} \|V_s(0)\| + \nu\varepsilon^{-2} \int_0^T   e^{-\mu(T-S) \varepsilon^{-2}} \cO(r_s)~\txtd S
\\&
+ \varepsilon^{2\alpha} \int_0^T  e^{- \mu (T-S)\varepsilon^{-2}} (T-S)^{-\alpha} [\cO(\varepsilon^{-1}(r_c+r_s))] ~\txtd S \nonumber\\
 &\leq  e^{- \mu T\varepsilon^{-2}} \|V_s(0)\| + \cO(\varepsilon)  (r_c+r_s). \label{e:boundVs}
\end{align}
The last bound was obtained using the substitution $y=\frac{T-S}{\varepsilon^{-2}}$ in the last integral to conclude that
\begin{align*}
    \varepsilon^{2\alpha} \int_0^T (T-S)^{-\alpha} e^{(T-S)\varepsilon^{-2}\mu}  \txtd S =\int_0^{T\varepsilon^{-2}} e^{-\mu y} y^{-\alpha} \varepsilon^{2}~\txtd y 
    &= \cO(\varepsilon^2).
\end{align*}

Thus for $\varepsilon$ sufficiently small  
\[
\sup_{[0,\tau^v]} \|V_s(T)\| \leq \frac12r_s
\quad\text{and}\quad 
\sup_{[T_\varepsilon,\tau^v]} \|V_s(T)\| =\cO(\varepsilon(r_c+r_s))
\]
provided that $r_s \geq3\|V_s(0)\|$.  

Moreover, we can modify \eqref{e:boundVs} to obtain
\begin{equation}
    \label{e:lucky2}
    \|V_s(T)\| \leq  e^{- \mu T\varepsilon^{-2}} \|V_s(0)\| 
    +  \cO(\varepsilon^{2\alpha-1}) \int_0^T  e^{- \mu (T-S)\varepsilon^{-2}} (T-S)^{-\alpha} \|V_c(S)\| ~\txtd S
    +\cO(\varepsilon) r_s. 
\end{equation}

\paragraph{2nd step} We show that $V_c=\cO(1)$ on $[0,T_0]$, where
\begin{align}\label{eq:vc}
\txtd V_c = [\nu\varepsilon^{-2}V_c + 2\varepsilon^{-1}B_c(U,V) ]~\txtd T.
\end{align}
We recall that we consider the case $\nu\varepsilon^{-2}\leq C$. 
Again we split $B_c(U,V)$ and observe that $B_c(U_c, V_c)=0$ by the assumption  $B_c(\cdot,\cdot)$ on $\cN\times \cN$. For $B_c(U,V)$ we use again $U=U_c +\varepsilon U_s$ and $V=V_c+V_s$. 
Thus 
\[
\txtd V_c = [\nu\varepsilon^{-2}V_c + 2\varepsilon^{-1}B_c(U_c,V_s) +2 B_c(U_s,V_c)+ 2B_c(U_s,V_s) ]~\txtd T.
\]
As before, we recall that $U_c=\cO(1)$  and $U_s=\cO(\varepsilon^{0-})$ on $[0,T_0]$ due to Theorem~\ref{thm:approximation}. Thus using \eqref{e:boundVs}
\[
\int_0^T B_c(U_s,V_s)~ \txtd S = 
\int_0^T B_c(\cO(\varepsilon^{0-}), e^{-\mu T \varepsilon^{-2}}\cO(1)+\cO(\varepsilon^{1-})(r_s+r_c)  )~\txtd S      =  \cO(\varepsilon)(r_s+r_c).  
\] 
Similarly, using additionally \eqref{e:lucky2} we  obtain
\begin{align*}
\varepsilon^{-1}\int_0^T B_c(U_c,V_s) ~\txtd S 
= & 
\cO(\varepsilon^{-1}) \int_0^T e^{-\mu S\varepsilon^{-2}} ~\txtd S    +\cO(1) r_s \\&
+ 
 \cO(\varepsilon^{2\alpha-2}) \int_0^T \int_0^S  e^{- \mu (S-s)\varepsilon^{-2}} (S-s)^{-\alpha} \|V_c(s)\| ~\txtd s~\txtd S 
\end{align*}
Using Fubini we obtain
\[ 
\varepsilon^{-1}\int_0^T B_c(U_c,V_s) ~\txtd S 
= \cO(1) \int_0^T \|V_c(S)\| \txtd S  + \cO(1) r_s .
\]

Finally, using Corollary \ref{cor:lucky} and $\sigma\leq \varepsilon^2$
\[
\int_0^T B_c(U_s,V_c)~ \txtd S 
     =\cO(1) (1+ \varepsilon^{-1} \|\tilde{Z}_\varepsilon\|) \int_0^T\|V_c(S)\|~\txtd S.   
\] 
Thus from~\eqref{eq:vc} we have on $ [0,\tau^v]$ for small $\varepsilon$
\[
\|V_c(T)\| 
\leq  \cO(1) (1+ \varepsilon^{-1} \|\tilde{Z}_\varepsilon\|)\int_0^T \|V_c(S)\|~\txtd S 
+\cO(1) r_s + \cO(\varepsilon) r_c .
\]
 Using Gronwall's inequality 
 and $\int_0^{T_0}\|\tilde{Z}_\varepsilon(S)\|\txtd S = \cO( \varepsilon)$
 we have
\[ \sup_{[0,\tau_v]}\|V_c(T)\| = \cO(1)\|V_c(0)\|+\cO(1)r_s+\cO(\varepsilon) r_c < r_c
\]
in case we fix 
 $r_c \gg\|V_c(0)\|$ and $\varepsilon\ll1$. In this case, based on the bounds obtained above, we can remove the stopping time and infer that
\[ V_s=\cO(\varepsilon) \text{ on } [T_\varepsilon, T_0] \text{ and } V_c=\cO(1) \text{ on } [0,T_0]. \]


\paragraph{3rd step} In order to show that $V_c$ solves a linearized AE around $U_c$,   we first use as before the It\^o-trick to replace 
\[V_s = -2\varepsilon A_s^{-1}B_s(U_c,V_c) +\cO(\varepsilon^2) \]
in $B_c(U_c,V_s)$. Replacing $U_s$ is done similarly below.

Recall that 
\[ \txtd V_c  =[\nu\varepsilon^{-2} V_c + 2\varepsilon^{-1}  B_c(U,V) ]~\txtd T.  \]
Splitting
\[ B_c(U,V)=B_c(U_c+\varepsilon U_s,V_c+V_s) = \varepsilon B_c(U_s,V_c) + B_c(U_c,V_s) + \varepsilon B_c(U_s,V_s) \] 
we see that we have to replace $V_s$ and $U_s$ above in order to get an equation for $V_c$ depending only on $U_c$ and $V_c$.  
Therefore we compute using It\^o's formula
\[
\txtd B_c(U_c,A_s^{-1}V_s) 
 = B_c(\txtd U_c,A_s^{-1}V_s)+   B_c(U_c,A_s^{-1} \txtd V_s),
\]
where from~\eqref{uc} we know that (up to a residual term $\text{Res}$ of order $\cO(\varepsilon^{1-})$)
\[
~\txtd U_c = \cO(1)\txtd T +  \cO(1) \txtd W_c +\txtd \text{Res}.
\]
Moreover on $[0,T_0]$ we have
\[ \txtd V_s = [\varepsilon^{-2} A V_s + \nu\varepsilon^{-2} V_s + 2\varepsilon^{-1} B_s(V_c,U_c)   +2\varepsilon^{-1} B_s(U_c,V_s) + 2 B_s(U_s,V_c)  + 2B_s(U_s,V_s)]~\txtd T. \]
Using  \eqref{e:boundVs} for $V_s$ but now on $[0,T_0]$, $U_c=\cO(1)$, $V_c=\cO(1)$ and $U_s=\cO(\varepsilon^{0-})$ on $[0,T_0]$, we get 
\[ \txtd V_s = [\varepsilon^{-2} A V_s + 2\varepsilon^{-1}B_s(V_c,U_c) +\cO(\varepsilon^{-1})e^{-T\mu \varepsilon^{-2}} +\cO(\varepsilon^{0-})  ]~\txtd T. \]

The relevant terms arise from the following computation
\begin{align*}
B_c(U_c,A^{-1}_s\txtd V_s)&= [B_c(U_c,  \varepsilon^{-2} V_s +2\varepsilon^{-1} A^{-1}_s B_s(V_c,U_c) + \cO(\varepsilon^{-1}) e^{-T\mu \varepsilon^{-2}} +\cO(\varepsilon^{0-})]\txtd T
\end{align*}
Summarizing  all the terms of higher order 
we obtain on $[0,T_0]$
\[
B_c(U_c,A_s^{-1}V_s) = 
\varepsilon^{-2} \int_0^T B_c(U_c,V_s)~\txtd S +2 \varepsilon^{-1} \int_0^T B_c(U_c,A_s^{-1}B_s(V_c,U_c))~\txtd S + \cO(\varepsilon^{0-}).     
\]
Thus 
\[\varepsilon^{-1}
\int_0^T B_c(U_c,V_s)~\txtd S= - 2 \int_0^T B_c(U_c,A_s^{-1}B_s(V_c,U_c))~\txtd S + \cO(\varepsilon^{1-}).
\]
For the next argument where we replace $U_s$, we have to make sure that the stochastic integral  
\[ \int_0^T B_c(\cO(1)\txtd W_c(S) , A^{-1}_sV_s(S)) \]
is well-defined and $\cO(1)$ and similarly for any other $\cO$-term. This follows as in Lemma~\ref{s:integrals}. More precisely, since $(W_c(T))_{T\in[0,T_0]}$ is a Wiener process on $\cN$ with covariance operator $Q_c$ we have by the Burkholder-Davis-Gundy inequality
\begin{align*}
    \E \sup\limits_{T\in[0,T_0]} \| \int_0^{T} B_c(\cO(1)~\txtd W_c(S), A^{-1}_sV_s(S))\|^P_Y &\leq C \Big( \E \int_0^{T_0} \| B_c( \cdot, A^{-1}_s V_s(S)Q^{1/2} )\|^2_{\cL_2(X,Y)}~\txtd S\Big)^{p/2}\\
    &\leq  C \Big( \sum\limits_{k=1}^\infty \|Q^{1/2}_c e_k \|^2_X \Big)^{p/2} \E\sup\limits_{T\in[0,T_0]} \|A^{-1}_s V_s(T)\|^p_X.
\end{align*}
This proves the claim by Chebyshev's inequality given that $A^{-1}_s$ is a bounded operator on $X$ and~\eqref{e:boundVs} to bound $V_s$.

We now use again the It\^o-trick in order to  similarly replace 
\[U_s = -\varepsilon A_s^{-1}B_s(U_c,U_c) +\cO(\varepsilon^{2-}) \]
in $B_c(U_s,V_c)$. To this aim, as before we consider
\[ 
\txtd B_c(A_s^{-1}U_s,V_c)  = B_c(A_s^{-1}\txtd U_s,V_c) + B_c(A_s^{-1}U_s,\txtd V_c).
\] 
We have from~\eqref{ito:us} and the bounds from Theorem \ref{thm:approximation}
\[
\txtd U_s =  [\varepsilon^{-2}A_sU_s 
+ \varepsilon^{-2} B_s(U_c,U_c) + \cO(\varepsilon^{-1-})]\txtd T + \cO(\varepsilon^{-1}) \txtd W_s   
\]
and from~\eqref{linear:v} since $U=U_c+\varepsilon U_s$ and $V=V_c+V_s$
\[
\txtd V_c = [\nu \varepsilon^{-2} V_c + 2  B_c(U_s,V_c) + 2 \varepsilon^{-1} B_c(U_c,V_s ) + 2B_c(U_s,V_s) ]\txtd T.   
\]
Thus using all the bounds  on $U$ and $V$ and Lemma~\ref{s:integrals} for the  stochastic integral
\begin{align*}
\cO(\varepsilon^{0-}) &= \int_0^T B_c(A_s^{-1}\txtd U_s,V_c)~\txtd S + \int_0^T B_c(A_s^{-1}U_s,\txtd V_c)~\txtd S\\
&= \int_0^T B_c(  [\varepsilon^{-2}U_s 
+ A_s^{-1}[\varepsilon^{-2} B_s(U_c,U_c) + \cO(\varepsilon^{-1-})],V_c)\txtd S +  \int_0^T \cO(\varepsilon^{-1}) \txtd W_s(S),V_c) 
\\ & \quad + \int_0^T B_c(A_s^{-1}U_s, [\nu \varepsilon^{-2} V_c + 2B_c(U_s,V_c) + 2 \varepsilon^{-1} B_c(U_c,V_s ) + 2B_c(U_s,V_s)])~\txtd S\\
&= \varepsilon^{-2} \int_0^T B_c(U_s, V_c) \txtd S
+  \varepsilon^{-2} \int_0^T B_c(A_s^{-1} B_s(U_c,U_c),V_c)\txtd S  \\& \quad +2 \varepsilon^{-1}   \int_0^T B_c(A_s^{-1}U_s,  B_c(U_c,V_s ))~\txtd S+  \cO(\varepsilon^{-1-}).
\end{align*}
Now we used \eqref{e:boundVs} for the term involving $V_s$, which gives an additional $\varepsilon$ leading to 
\[  \int_0^T B_c(U_s,V_c)~\txtd S = - \int_0^T B_c(A^{-1}_s B_s(U_c,U_c),V_c)~\txtd S +\cO(\varepsilon^{1-}).  \]

This allows us to replace $U_s$ in $B_c(U_s, V_c)$ as stated above. We finally obtain on $[0,T_0]$
\begin{align*}
 V_c(T) & = V_c(0)+\int_0^T[\nu\varepsilon^{-2} V_c + 2\varepsilon^{-1}  B_c(U,V) ]~\txtd S\\
&=V_c(0)+\int_0^T[\nu\varepsilon^{-2} V_c + 2 B_c(U_s,V_c) + \varepsilon^{-1} B_c(U_c,V_s) + \cO(\varepsilon^{1-})]~\txtd S\\
&=V_c(0)+\int_0^T[\nu\varepsilon^{-2} V_c - 2 B_c(A^{-1}_s B_s(U_c,U_c),V_c)  - 4 B_c(U_c,A_s^{-1}B_s(V_c,U_c))]~\txtd S+ \cO(\varepsilon^{1-}).
\end{align*}

 \begin{remark}
    Note that we now have $B_c(U,V)=B_s(U_c,V_s) + \varepsilon B_c(V_s,U_s) + \varepsilon B_c(U_s,V_c)$ since $B_c(U_c,V_c)=0$ and all the other terms are higher order. In the {\bf 3rd step} the term $B_s(U_c,V_c)$ entered the computation.
 \end{remark}

Thus we have the following equation for $V_c$
\begin{align*}
\txtd V_c &= 
[\nu\varepsilon^{-2} V_c-2B_c(V_c,A_s^{-1}B_s(U_c,U_c)) - 4B_c(U_c,A_s^{-1}B_s(U_c,V_c))]~\txtd T + \txtd R_V\\
&=[\nu\varepsilon^{-2} V_c+D\cF_c(U_c)V_c]~\txtd T + \txtd R_V.
\end{align*}
with $R_V=\cO(\varepsilon^{1-})$. 

\paragraph{4th step}
We now compare $V_c$ to the solution $\varphi$ of the linearized amplitude equation~\eqref{ae} around an arbitrary solution $a$ given by
\[
\txtd \varphi = 
[\nu\varepsilon^{-2} \varphi +D\cF_c(a)\varphi]~\txtd T.
\]
We show that on $[T_\varepsilon, T]$
\[ \|V_c-\varphi\|=\cO(\varepsilon).  \] 
We have 
\[ \txtd V_c = [\nu\varepsilon^{-2} V_c + D\cF_c(U_c)V_c]~\txtd T +\txtd R_V \]
and 
\[
\txtd\tilde{\varphi} = [\nu\varepsilon^{-2} \tilde{\varphi} + D\cF_c(a)\tilde{\varphi}]~\txtd T +\txtd R_V, 
\]
for a small error term $R_V=\cO(\varepsilon^{1-})$. Taking the difference we get
\begin{align*}
    \txtd (V_c-\tilde{\varphi}) = [\nu\varepsilon^{-2} (V_c-\tilde{\varphi}) + D \cF_c(U_c) (V_c-\tilde{\varphi}) + (D\cF_c(U_c) -D\cF_c(a))\tilde{\varphi}]~\txtd T.
\end{align*}
Since $\langle D\cF_c(a) b, b \rangle\leq 0$ for all $a,b\in \cN$, we get
\begin{align*}
    \frac{1}{2}\partial_T \|V_c-\tilde{\varphi}\|^2 &=\nu \varepsilon^{-2} \|V_c-\tilde{\varphi}\|^2 +\langle D\cF_c(U_c) (V_c-\tilde{\varphi}), V_c-\tilde{\varphi}  \rangle +\langle [D\cF_c(U_c) -D\cF_c(a)]\tilde{\varphi},V_c-\tilde{\varphi} \rangle\\
    & \leq \nu \varepsilon^{-2} \|V_c-\tilde{\varphi}\|^2 +\frac{1}{2} \| D\cF_c(U)-D\cF_c(a)\|^2 \|\tilde{\varphi}\|^2 +\frac{1}{2}\|V_c-\tilde{\varphi}\|^2. 
 \end{align*}
 By Gronwall's lemma we get for some universal constants $c,C>0$ and $T\in[0,T_0]$ for that
 \[ \|V_c(T)-\tilde{\varphi}(T)\|^2 \leq c\| V_c(0)-\tilde{\varphi}(0)\|^2 + C \int_0^T e^{c(T-\tau)} \|D\cF_c(U_c(S))-D\cF_c(a(S))\|^2\|\tilde{\varphi}(S) \|^2~\txtd S.  \]
By Theorem~\ref{thm:approximation} we know that $U_c-a=\cO(\varepsilon^{1-})$ on $[0,T_0]$.~Using this together with the local Lipschitz continuity of $D\cF_c$
\[
  \| D\cF_c(a)-D\cF_c(U_c)\| \leq C\|a+U_c\|\|a-U_c\| 
\]
that can be easily verified, proves the statement.

 To finalize the proof we need a bound on 
\[
\sup_{[0,T]}\| \varphi-\tilde\varphi\| \leq C\sup_{[0,T]}\|R_V\|
\]
which is trivial using Gronwall's inequality since $a$ is $\cO(1)$ according to Lemma~\ref{o:ae} and $\cF_c$ is a cubic, thus $D\cF_c$ quadratic. As the equation for $\varphi$ is linear, this is exactly the proof that $\tilde\varphi$ is bounded by the error given by $R_V$.

\qed
\end{proof}



\subsection{Approximation of finite-time Lyapunov exponents}


We give a fairly general statement on the approximation, that does not require any structure of the equations, or any approximation result. These are highly needed when evaluating the constants $K_N$ and $K_X$ introduced below.

\begin{theorem}\label{thm:lyapappr}
Let $U$ be the solution of the SPDE on the slow time-scale $T$ as in Section~\ref{sec:linearization} and let $a$ be a given solution of the AE~\eqref{ae}. 
We denote by $V$ the linearization around $U$ and $\varphi$ the linearization of the amplitude equation~\eqref{ae} around $a$ such that  $\varphi(0)=P_cV(0)$. We denote by $\lambda^U_T$ the FTLE of the SPDE and by $\lambda^a_T$ the FTLE of the AE. 
We further denote 
\[K_X(T):=\sup\limits_{\|V(0)\|=1}\|V(T)-\varphi(T)\| 
\quad\text{and}\quad
K_\cN(T):= \sup\limits_{\|V(0)\|=1, V(0)\in\cN} \|V(T)-\varphi(T)\|. 
\]
Then 
\[
 -\frac{1}{T} C K_\cN(T) e^{-T\lambda^a_T} \leq 
\lambda^U_T - \lambda^a_T \leq   \frac{1}{T} K_X(T)e^{-T\lambda^a_T},  
\]
where the lower bound only holds if $K_\cN(T) e^{-T\lambda^a_T}< 1/2$.
\end{theorem}

\begin{proof}

For the upper bound we use the definition of the FTLE to get that
   \begin{align*}
       \lambda^U_T
    &=\sup\limits_{\|V(0)\|=1} \frac{1}{T} \log (\|V(T)\|)\\
       & \leq \sup\limits_{\|V(0)\|=1} \frac{1}{T} \log (\|V(T)-\varphi(T)\| +\|\varphi(T)\| ). 
   \end{align*}

Furthermore, since  
\[ e^{T\lambda^a_T} = \sup\limits_{\|\varphi(0)\|=1} \|\varphi(T)\| \]
and using that $\varphi$ solves a linear SDE, we infer that  
\[ 
\|\varphi(T)\| \leq \|\varphi(0)\|e^{T\lambda^a_T}.
\]
Thus 
\[
 \lambda^U_T\leq  \frac{1}{T} \log ( K_X(T) +\|P_c\| e^{T\lambda^a_T}  )=\frac{1}{T} \log ( e^{T\lambda^a_T} (\|P_c\| + K_X(T) e^{-T \lambda^a_T})   ).
\]
In conclusion,
\begin{align*}
 \lambda^U_T & \leq  \lambda^a_T  + \frac{1}{T} \log \Big( K_X(T)e^{-T\lambda^a_T}  +\|P_c\|   \Big) \\ 
  & \leq  \lambda^a_T  + \frac{1}{T} \log (\|P_c\|) + \frac{1}{T} K_X(T)e^{-T\lambda^a_T} . 
\end{align*}
Note that by our assumptions $P_c$ is an orthogonal projection, thus its norm is $1$.

For the lower bound, we fix a maximizing sequence of $\varphi_n(0)=V_n(0) \in \cN $ such that 
\[
\|\varphi_n(T)\| \to e^{T\lambda^a_T} \qquad \text{for } n\to\infty.
\]
Then 
  \begin{align*}
       \lambda^U_T= \frac{1}{T} \sup\limits_{\|V(0)\|=1} \log (\|V(T)\|)
    & \geq \frac{1}{T} \log (\|V_n(T)\|)\\
    & \geq  \frac{1}{T} \log (\|\varphi_n(T)\| -K_\cN(T) )
   \end{align*}
Thus in the limit
 \begin{align*}
\lambda^U_T  
&\geq  \frac{1}{T} \log (e^{T\lambda^a_T} -K_\cN(T) ) \\
&= \lambda^a_T + \frac{1}{T} \log (1-K_\cN(T) e^{-T\lambda^a_T} ) .
\end{align*}
Using the inequality $\ln(1-x)\geq -c x$ 
for $0\leq x\leq 1/2$ we get
provided $K_\cN(T) e^{-T\lambda^a_T}\leq\delta<1/2$, 
\[
\lambda^U_T \geq \lambda^a_T - \frac{1}{T}C_\delta K_\cN(T) e^{-T\lambda^a_T}.
\]  
Here
$C_\delta =\inf\{ 1/(1+\xi)\ :\  \xi \in (-\delta,0) \} = 1/(1-\delta)$. This proves the lower bound.

\end{proof}

\begin{remark}\label{rem:K} Let us comment on the setting of the AE equation
\begin{itemize}
    \item [1)] Note that the AE~\eqref{ae} does not depend on $\varepsilon$, so in particular, $\lambda^a_T$ does not depend on $\varepsilon$ as required.
    \item [2)] In applications, the upper bound in the previous theorem is used to prove stability and the lower bound is used for instability. 
    \item [3)] Note that 
    \[ \|V(T)-\varphi(T)\| \leq \|V_s(T)\| +\| V_c(T)-\varphi(T)\|, \]
    but
    \[\sup\limits_{[T_\varepsilon,T]} \|V_s(T)\|=\cO(\varepsilon). \]
    This means that provided the initial conditions are $u(0)= \varepsilon a_0 + \cO(\varepsilon^2)$ we have 
       $K_{\cN}=\cO(\varepsilon^{1-})$ but for $K_X$ since $\|V_s(T)\|=\cO(\varepsilon)$ for $T\in[T_\varepsilon,T_0]$  this is not the case. 
       To be more precise, using \eqref{e:boundVs} we have for an arbitrarily small $\kappa>0$
       \[
       K_X(T) \leq C e^{-\mu T\varepsilon^{-2}} + C \varepsilon^{1-\kappa}.
       \]
\end{itemize}

\end{remark}

\subsection{Main result}

Let us finally rescale the result  to the scaling of AE.

\begin{theorem}\label{scale:ftle}
    Let $u$ be the solution of the SPDE~\eqref{spde} on the original time scale with $u(0)=u_0$ and $\varphi$ be the linearization of the amplitude equation~\eqref{ae} around $a$ with $a(0)=a_0$ such that $\varphi(0)=P_cV(0)$.
Then we get the following approximation of the FTLEs
\begin{align}\label{approx:ftle}
   -\frac{1}{T}C \varepsilon^2 K_{\cN}(T) e^{-T\lambda^a_T}(a_0) \leq \lambda^u_{T\varepsilon^{-2}}(u_0) -\varepsilon^2 \lambda^a_T(a_0)  \leq \frac{\varepsilon^2}{T}K_X(T)
e^{-T\lambda^a_T(a_0)}
\end{align}
where the lower bound only holds if $K_\cN(T) e^{-T\lambda^a_T(a_0)}< 1/2$.
\end{theorem}
\begin{proof}
    We recall that we consider the slow time scale $t=T\varepsilon^{-2}$. Therefore, rescaling the quantities we obtain the following relation between the FTLE for $u$ on the original time scale and for the FTLE for $u$ on the slow time scale
\begin{align*}
    \lambda_{T\varepsilon^{-2}}^u = \lambda_{T\varepsilon^{-2}}^{\varepsilon U(\epsilon^2 \cdot)}
    &=\frac{\varepsilon^2}{T}\sup\limits_{\|v(0)\|=1}\log (\|v(T/\varepsilon^2)\|)\\
    & = \frac{\varepsilon^3}{T}\sup\limits_{\|V(0)\|=\varepsilon^{-1}} \log \|V(T)\|\\
    &=\frac{\varepsilon^2}{T} \sup\limits_{\|V(0)\|=1} \log \|V(T)\| = \varepsilon^2\lambda_T^U.
\end{align*}
The statement follows from Theorem \ref{thm:lyapappr}.
\end{proof}\\


From Remark \ref{rem:K} and Theorem \ref{scale:ftle} we obtain immediately 

\begin{corollary}\label{corollar:ftle}
Fix $T_0>0$, $\alpha\in(0,1)$, and small $\kappa>0$. Moreover, consider initial conditions $a_0=a_0(\omega)=\cO(1)$ for the AE \eqref{ae}  and $u_0=\varepsilon a_0 +\varepsilon^2\psi_0$ with $\psi_0\in\cS$ and $\psi =\cO(1)$ for the SPDE \eqref{spde}.  For any probability $p\in(0,1)$ there is a set $\Omega_p$ with $\P(\Omega_p)>p$, and $\varepsilon_0>0$ and a constant $C>0$ such that for all $T\in[\varepsilon^\alpha,T_0]$
  \[
   | \lambda^{u}_{T\varepsilon^{-2}} (\varepsilon a_0 + \varepsilon^2\psi_0) -\varepsilon^2 \lambda^{a}_T(a_0) | \leq C  \varepsilon^{3-\alpha-\kappa} 
 \]
 for all $\varepsilon\in(0,\varepsilon_0]$.
\end{corollary}


\section{Example: Bifurcations for one-dimensional kernel}\label{sec:bif}

    We emphasize that all approximating SDEs have a stable cubic nonlinearity, so we can use well-known results for the FTLEs obtained in~\cite{BlNe:23}. 

To simplify the presentation of the results we assume throughout the section that $\text{dim}(\cN)=1$. Now according to Assumption~\ref{b}
\[\cF_c(a)=-B_c(a,A^{-1}_sB_s(a,a))=-ca^3
\] 
for some constant $c$ if we identify  $\cN\equiv \R$ and know that $\cN=\text{span}\{e\}$. To 
be more precise
\[ \langle F(a), e\rangle = -\xi^3 \langle F(e), e\rangle<0 \]
for $a=\xi \cdot e$. Analogously we get for the derivative 
\[ \langle \txtD F(a)e , e \rangle =- 3 \xi^2 \langle \txtD F(e)e,e \rangle <0. \]

\subsection{Case $1\gg\sigma\approx\nu >0$}

\begin{theorem}{\em(Instability)}
For $\alpha\in(0,1/2)$, $T_0>0$, $a_0\in\cN$ and $\psi\in\cS$ we have
\[
\lambda^u_{T\nu^{-1}} (\sqrt{\nu} a_0+ \nu^2\psi) >0
\]
 with positive probability
for all $T\in [\nu^{\alpha}, T_0]$ and $\nu$ sufficiently small.
\end{theorem}

\begin{proof}
    Setting $\nu =\varepsilon^2$ we get the amplitude equation
\[ \txtd a = [a +2 \cF_c(a)  ]~\txtd T +\frac{\sigma}{\nu}~\txtd \beta_{\sqrt{\nu}}(T), \]
where  $(\beta_{\sqrt{\nu}}(T))_{T\geq 0}$ is a rescaled Brownian motion. Furthermore we get the linearization
\[
\txtd \varphi = 
[ \varphi -{2}B_c(\varphi,A_s^{-1}B_s(a,a))-{4}B_c(a,A_s^{-1}B_s(a,\varphi))]~\txtd T.
\]

Therefore we get 
\begin{align*}
    \lambda^{a}_T(a_0) =\frac{1}{T} \log \|\varphi(T)\| \geq \frac{1}{T} \log \exp \Big( T- 6 \int_0^T a^2(S)~\txtd S  \Big) \geq 1 - 6\delta^2.
\end{align*}
Choosing $\delta:=\frac{1}{4}$ we get that $\lambda^{a}_T(a_0)>0$ on $\Omega_0$. 
 Thus we can conclude using Corollary~\ref{corollar:ftle} that on $\Omega_0\cap \Omega_p$ we have for $T\in [\varepsilon^{\alpha},T_0]$
\begin{align*}
    \lambda^{u}_{T\varepsilon^{-2}}(\varepsilon a_0 +\varepsilon^2 \psi ) \geq \varepsilon^2 \lambda^{a}_T(a_0)-C\varepsilon^{3-\alpha-\kappa}. 
\end{align*}
Since $\lambda^{a}_T(a_0)$ is positive on the set $\Omega_0\cap \Omega_p$ we get the positivity of $\lambda^{u}_{T\varepsilon^{-2}}(\varepsilon a_0)$ on $\Omega_0\cap \Omega_p$  choosing $C\varepsilon^{3-\alpha}$ small. 

Note finally that $\Omega_0\cap \Omega_p$ is a set of positive probability if $p$ is sufficiently close to $1$, for example larger than $1-\P(\Omega_0)$.
\end{proof}

\subsection{Case $1\gg\sigma\approx -\nu >0$}

\begin{theorem}{\em(Stability)}
For $\alpha\in(0,1/2)$ $T_0$, $a_0\in\cN$ and $\psi\in\cS$ we have
\[
\lambda^u_{T|\nu|^{-1}} (\sqrt{|\nu|} a_0+ |\nu|\psi) <0
\]
 with probability almost one
for all $T\in [|\nu|^{\alpha}, T_0]$ and $|\nu|$ sufficiently small.
\end{theorem}

\begin{proof}
    Setting $|\nu| =\varepsilon^2$ we get again  the amplitude equation
\[ \txtd a = [-a +2 \cF_c(a)  ]~\txtd T +\frac{\sigma}{\nu}~\txtd \beta_{\sqrt{|\nu|}}(T), \]
where  $(\beta_{\sqrt{|\nu|}}(T))_{T\geq 0}$ is a rescaled Brownian motion. Furthermore we obtain the linearization
\[
\txtd \varphi = 
[ -\varphi -{2}B_c(\varphi,A_s^{-1}B_s(a,a))-{4}B_c(a,A_s^{-1}B_s(a,\varphi))]~\txtd T.
\]
We get 
\begin{align*}
    \lambda^{a}_T(a_0) \leq \frac{1}{T} \log \exp \Big( -T- 6 \int_0^T a^2(S)~\txtd S  \Big) \leq -1. 
\end{align*}
 Thus we can conclude using Corollary~\ref{corollar:ftle} that on $ \Omega_p$ we have for $T\in [\varepsilon^{\alpha},T_0]$
\begin{align*}
    \lambda^{u}_{T\varepsilon^{-2}}(\varepsilon a_0+\varepsilon^2 \psi)\leq \varepsilon^2\lambda^{a}_T(a_0) +C\varepsilon^{3-\alpha-\kappa}.
\end{align*}
Since $\lambda^{a}_T(a_0)$ is almost surely negative we get the negativity of $\lambda^{u}_{T\varepsilon^{-2}}(\varepsilon a_0)$ on $\Omega_p$  choosing $C\varepsilon^{3-\alpha-\kappa}$ small.

\end{proof}

\subsection{Case $1 \gg |\nu| \gg \sigma$}
 Here again we fix $\varepsilon^2=|\nu|$ the amplitude equation is given by the ODE
 \begin{align}\label{ae:ode} \txtd a = [\sgn(\nu)a +2\cF_c(a) ]~\txtd T +  \frac{\sigma}{\nu}~\txtd \beta_{\sqrt{|\nu|}}(T).
 \end{align}

We have a small noise term and could follow the results of the previous two sections.
Alternatively, we can 
use large deviations principles (LDP) for FTLEs \cite{BBBE:25}, so compute FTLEs for the ODE 
\[
\txtd a = [\sgn(\nu)a +2\cF_c(a) ]~\txtd T
\]
and transfer it to the SDE by LDP. As we can rely on the deterministic amplitude equation (neglecting the small noise term), we do not need the support theorem and could prove a result for instability or stability (depending on the sign of $\nu$) also with probability almost one, depending on the initial condition.
We refrain from giving more details here.

\subsection{Case $1 \gg  \sigma\gg|\nu|$}

Here we set $\sigma=\varepsilon^2$ and the amplitude equation is given by the SDE
 \begin{align}
 \txtd a = [\frac{\nu}{\sigma}a +2\cF_c(a) ]~\txtd T + \txtd \beta_{\sqrt{\sigma}}(T).
 \end{align}
For $\nu=0$ we could use either Birkhoff's ergodic theorem \cite[Lemma 4.4]{BlNe:23} if the initial datum $a_0(\omega)$ is distributed according to the the invariant measure of the AE to show that $\lambda_T^a<0$ with high probability. For general initial data $a_0$ we can rely on \cite[Lemma 4.2]{BB:25} to show that the solution of the AE does not spend too much time in zero from which we can infer again that $\lambda_T^a<0$ with high probability.~In both cases we can carry over this stability to the SPDE using Corollary \ref{corollar:ftle}.

For the case $\nu\not=0$ we need an approximation result similar to Theorem \ref{thm:approximation} that relates results of $\nu=0$ to $\nu\neq 0$ but small, which we will not state here.

\section{Examples}\label{sec:app}

\subsection{Burgers equation}\label{sec:burgers}
We consider the Burgers equation  subject to Dirichlet boundary condition on $[0,\pi]$ given by
\begin{equation}\label{burgers}
    \partial_t u = (\partial^2_x +1) u +  \nu u + u \partial_x u +\sigma \partial_t W.
\end{equation}

In this case $X=L^2([0,\pi])$ and we have the orthonormal basis $e_k(x)=\sqrt{\frac{2}{\pi}}\sin (k x)$ of eigenfunctions of $A=\partial_x^2+1$, corresponding to the eigenvalues $\lambda_k=k^2-1$ for $k\geq 1$  and $\cN=\text{span}\{\sin x\}$.
The statements of Assumption \ref{a} are easy to verify for $A$ in $X$ and $Y=H^{1/4}([0,\pi])$.

The quadratic nonlinear term is given by
\[ B(u,v)=\frac{1}{2}\partial_x(uv). \]
On $\cN$ we immediately verify that for $u=c\sin x$ for some $c>0$
\[ P_c B(u,u)=\frac{1}{2}c^2 P_c(\partial_x \sin^2 x)=\sin x \cos x =0. \]
Furthermore we have
\begin{align*}
    2\|B(u,v)\|_{H^{-1}} &=\|\partial_x (uv)\|_{H^{-1}} \leq \|u v\|_{L^2}\\
    & \leq \|u\|_{L^4} \|v\|_{L^4}\leq C \|u\|_{H^{1/4}} \|v\|_{H^{1/4}},
\end{align*}
using the embedding $H^{1/4}([0,\pi])\hookrightarrow L^4([0,\pi])$. Therefore, Assumption~\ref{b} is satisfied for $X=L^2([0,\pi])$ and $Y=H^{1/4}([0,\pi])$.
In this situation  one can check that 
\begin{align}\label{nl:b}
    \cF_c(a)=-P_cB(a, A^{-1}_s B_s(a,a)) =-\frac{1}{24} a^3 \sin(x).
\end{align}

Furthermore, similar to~\cite{BlMo} we assume that the noise is given by a cylindrical Wiener process having the representation
\begin{align}\label{noise:b}
    W(t)= \sum\limits_{k=1}^\infty \alpha_k \beta_k(t) e_k,
\end{align}
 where $(\alpha_k)_{k\geq 1}$ is a sequence of bounded real numbers and $(e_k)_{k\geq 1}$ is the ONB specified above.~The covariance operator $Q$ is Fourier diagonal and satisfies $Q e_k =\alpha^2_k e_k$, which implies that $W_c(t)=\alpha_1\beta_1(t)e_1$.~Plugging this in~\eqref{sum:c} and~\eqref{sum:s} verifies Assumption~\ref{n}. We also note the $A_s^{-1}$ is a trace class operator, which shows that $W_s$ is an $X$-valued process. Moreover, Assumption \ref{o:z} is well-known in the literature~\cite{DaPZ:92} and straightforward to check.\\

In conclusion, we obtain the amplitude equation
\begin{align}\label{ae:b}
    \txtd a = [\frac{\nu}{\varepsilon^2}  a -\frac{1}{12} a^3 ]~\txtd T+\frac{\sigma}{\varepsilon^2} \alpha_1 \txtd \beta_{\varepsilon}(T) 
\end{align}
for which the bifurcation analysis made Section~\ref{sec:bif} applies.

\subsection{Kuramoto-Sivashinsky}

Consider the Kuramoto-Sivashinsky equation
\[
\partial_t u= - \partial_x^4 u + \nu u + |\partial_x u|^2 + \sigma \partial_t W 
\]
on the interval $[0,2\pi]$ subject to periodic boundary conditions.

The operator $A=-\partial_x^4$ satisfies Assumption \ref{a}
as it is Fourier diagonal with $\cN=\text{span}\{1\}$. 
The spaces are $X=L^2([0,2\pi])$ and $Y=H^{5/4}_{\text{per}}([0,2\pi])$.~One can readily check that $B(u,v)=\partial_xu\cdot \partial_xv$ satisfies Assumption \ref{b}.~Moreover,  both Assumptions \ref{n} and \ref{o:z} are easy to check in a similar way as for Burgers equation with the same choice of noise as in~\eqref{noise:b}.~Unfortunately, $\cF_c(a)=0$  leading to the linear amplitude equation 
\begin{align}
    \txtd a = \frac{\nu}{\varepsilon^2}  a~\txtd T +\frac{\sigma}{\varepsilon^2} \alpha_1~ \txtd \beta_{\varepsilon}(T). 
\end{align}
Thus the bifurcation analysis presented in Section~\ref{sec:bif} does not apply in this case, but the Lyapunov exponents $\lambda_T^a$ are easy to calculate.  

\subsection{Rayleigh-B\'enard}\label{sec:rb}
This model consists of a Navier-Stokes equation coupled to a heat equation, it does not fit into our setting, as the 
the operator $A$ is not self-adjoint but it has a compact resolvent. We have a basis of eigenfunctions but this is not orthonormal, in particular $P_c$ is no longer an orthogonal projection with norm one.~Moreover, $\cN$ might not be one-dimensional and thus the bifurcation analysis presented in Section~\ref{sec:bif} does not apply in this case, as the FTLEs for higher dimensional SDEs are in general not easy to compute.

\section{Appendix}
We state some straightforward results which we use in the proof of Theorem~\ref{thm:approximation} and Theorem~\ref{thm:linear}.
\begin{lemma} \label{lem:SDE-contdep} Let $f$ be locally Lipschitz and $a$ bounded. We consider the SDEs on $\cN$
\begin{align*}
    \txtd a =f(a)~\txtd T + \txtd B(T)\\
    \txtd b =f(b)~\txtd T + \txtd R +\txtd B(T),
    \end{align*}
    where $(B(T))_{T\geq 0}$ is a Brownian motion on $\cN$ and $R$ is small. Then for all $K>0$ there exists a constant $C>0$ such that for all $\delta\in(0,1)$ such that $\|a\|\leq K$ and $\|R\|\leq \delta$ we have that 
    \[ \|a(T)-b(T)\|\leq C \delta. \]
    \end{lemma}
    \begin{proof}
        This is straightforward.     Setting $\tilde{b}:=b-R$ we get using the local Lipschitz continuity of $f$ together with the boundedness of $a$ and $R$ that
    \end{proof}
    \begin{align*}
        \|a(T)-b(T)\| &\leq \|a(T) -\tilde{b}(T)\| + \|R(T)\|\\
        & \leq   \int_0^T \| (f(a(S)) - f(\tilde{b}(S)+R(S))\|
        ~\txtd S   + \|R(T)\|\\
        & \leq \int_0^T \|a(S)-b(S) \|~\txtd S +  \|R(T)\|.
    \end{align*}
   Gronwall's lemma proves the statement.

\begin{lemma}
Let $a$ in $\cN$ be a solution of SDE 
\[
\txtd a = f(a)~\txtd T+ \cO(1)~\txtd B(T) 
\] 
with $a(0)=\cO(1)$ and $\langle  f(a+b),a\rangle \leq C\|a\|^4-\|b\|^4$ for all $a,b\in \cN$. Then for all $T_0>0$ one has  $a=\cO(1)$ on $[0,T_0]$.
\end{lemma}
\begin{proof}
    This follows by subtracting the noise, i.e. $\tilde{a}:=a-B$ leading to 
    \begin{align*}
        \txtd \tilde{a} =f(\tilde{a} +B)~\txtd T,
    \end{align*}
    entailing that
    \begin{align*}
        \frac{1}{2} \partial_T \|\tilde{a}\|^2 = \langle f(\tilde{a}+B), \tilde{a}\rangle~\txtd T \leq C \|\tilde{a}\|^4 - \|B\|^4.
    \end{align*}
 Again, an application of Gronwall's lemma proves the statement. 
\end{proof}

\end{document}